\documentclass[10pt]{article}
\usepackage{amsmath,amssymb,amsthm}
\usepackage{graphicx} 

\newtheorem{thm}{Theorem}
\newtheorem{prop}{Proposition}

\begin{document}

\title{The hafnian of Toeplitz matrices of a special type, perfect matchings and Bessel polynomials}
\author{Dmitry Efimov\thanks{e-mail: dmefim@mail.ru}\\ Institute of Physics and Mathematics,\\ Komi Science Centre UrD RAS,\\
Syktyvkar, Russia}
\date{}

\maketitle

\begin{abstract}
We present a simple and convenient analytical formula 
for efficient exact computation of the hafnian of Toeplitz matrices of a special type.
An interpretation of the obtained results is given in the language of perfect matchings 
and Bessel polynomials.
\end{abstract}

%------------------------------------------------------------------------------------------
\section*{Introduction} 
Let $A=(a_{ij})$ be a symmetric matrix of order $n=2m$ over a commutative associative ring.
Its hafnian is defined as 
$$
  \textrm{Hf}(A)=\sum_{(i_1i_2|\dots|i_{n-1}i_n)} a_{i_1i_2}\dots a_{i_{n-1}i_n},
$$
where the sum runs over all partitions of the set $\{1,2,\dots, n\}$
into disjoint pairs $(i_1i_2),\dots,(i_{n-1}i_n)$  
up to the order of pairs, and the order of elements in each pair.
So, for example, if $n=4$ then $\textrm{Hf}(A)=a_{12}a_{34}+a_{13}a_{24}+a_{14}a_{23}$.
Equivalently, one can define the hafnian as
$$
  \textrm{Hf}(A)=\frac{1}{m!2^m}\sum_{\sigma\in S_n} a_{\sigma(1),\sigma(2)}\dots a_{\sigma(n-1),\sigma(n)},
$$
where the sum runs over all permutations of the set $\{1,2,\dots, n\}$.
Note that diagonal elements of the matrix are not present in the definition of the hafnian.
We will take them equal to zero for convenience.
The hafnian was introduced by E.R. Caianiello in one of his works on quantum field theory \cite{Cai1}.
He gave the name to the new matrix function in honor of Copenhagen (Lat. Hafnia),
the place where the idea of this mathematical concept first occurred to him.
By this name he also emphasized  connection with Pfaffian introduced by A. Cayley in 19th century,
from which the hafnian differs only by the signs of some components.
Later it became clear that the hafnian also has a useful combinatorial property
related to solving an important problem in graph theory: 
if $M$ is the adjacency matrix of an unordered graph with even number of vertices, 
then $\textrm{Hf}(M)$ equals the total number of perfect matchings of the graph.

Unfortunately  the widespread use of the hafnian is limited by the fact that
there are no effective algorithms for its calculation in the general case.
Thus, in the recent work \cite{Titan}, the currently fastest exact algorithm 
to compute the hafnian of an arbitrary complex $n\times n$ matrix is described.
It runs in $O(n^32^{n/2})$ time.
And numerical benchmarks on the Titan supercomputer
(the 7th place in the Top500 ranking as  June 2018) 
indicated that it would require  about a month and a half to compute 
the hafnian of a randomly generated $100\times 100$ complex matrix using this algorithm.

Since in the general case calculation of the hafnian has a high computational complexity,
the problem is actual of finding efficient analytical formulas expressing the hafnian for special classes of matrices.
Thus, in many important cases (e.g., when considering adjacency matrices of planar graphs) 
one can reduce calculation of the hafnian of a given matrix to a much more efficient calculation of the Pfaffian 
for another matrix, associated with the former one by certain simple transformations \cite{Kuper}.
Recall that a matrix is called Toeplitz if all elements of every of its diagonal parallel to the main one are the same.
The paper \cite{Moshe} presents an algorithm for calculating the hafnian 
of banded Toeplitz matrices of order $n$ and bandwidth $m$ which runs in $O(2^{3m}\log n)$ time.

In this work we obtain an efficient analytical formula for calculating the hafnian 
of Toeplitz matrices of a special type, different from the one mention above.
In a special case this formula reduces to computation of the value of the Bessel polynomial 
of the corresponding degree at a certain point. 
Using this formula, one can calculate the hafnian in linear time.

%--------------------------------- The main part ------------------------------------------
\section*{The main part}
To begin with, consider two properties of the hafnian.
The first property is quite obvious.

\begin{prop}\label{prop1}
 Let $A$ be a symmetric matrix of order $2m$ over a commutative associative ring $R$, and $c\in R$. Then
 \begin{equation}\label{16.10.18_1}
  \mathrm{Hf}(cA)=c^m\mathrm{Hf}(A).
 \end{equation}
\end{prop}

Let $Q_{k,n}$ denote the set of all (unordered)  $k$-element subsets of the set $\{1,2,\dots,n\}$.
Let  $A$ be a matrix of order $n$ and $\alpha\in Q_{k,n}$. 
We denote by $A[\alpha]$ the submatrix of $A$ formed by the rows and columns of $A$ with numbers in $\alpha$, 
and by $A(\alpha)$ the submatrix of $A$ formed from $A$ by removing the rows and columns with numbers in $\alpha$.
The following property proved in \cite{Efimov}:

\begin{prop}\label{prop2}
Let $A$, $B$ be symmetric matrices of order $2m$. Then 
\begin{equation}\label{24.11_3}
 \mathrm{Hf}(A+B)=\sum_{k=0}^m \sum_{\alpha\in Q_{2k,2m}}\mathrm{Hf}(A[\alpha])\mathrm{Hf}(B(\alpha)),
\end{equation}
where $\mathrm{Hf}(A[\alpha])=1$ if $\alpha\in Q_{0,2m}$, 
and $\mathrm{Hf}(B(\alpha))=1$ if $\alpha\in Q_{2m,2m}$.
\end{prop}

To prove our main result, we need also one combinatorial property of path graphs.
Recall that {\itshape a path graph} is a graph that can be drawn so that all of its vertices and edges lie on a single straight line
and all neighboring vertices are adjacent. We denote the path graph with $n$ vertices by $P_n$ (Fig. \ref{Efimov1}).
\begin{figure}[!ht]
 \begin{center}
 \includegraphics[scale=0.9]{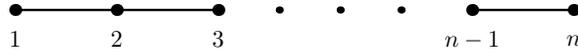}
 \caption{The path graph  $P_n$}\label{Efimov1}
 \end{center}
\end{figure}
Let $k$ be a non-negative integer less than $n/2$, 
and  let $E_n^k$ denote the number of ways to select $k$ edges in $P_n$ so that 
there are no two selected edges with a common vertex. 

\begin{prop}\label{08.02_1}
The value of $E_n^k$ equals the binomial coefficient $C_{n-k}^k$.
\end{prop}
\begin{proof}
It is obvious that one can select $k$ different edges in $P_n$ in $C_{n-1}^k$ ways.
Now we replace the given path graph by a new one according to the following rule:
we insert an additional edge after each of the first $k-1$ selected edges, counting left to right.
As a result, we obtain a path graph with $n+k-1$ vertices and $k$ selected edges,
any two of which have no common vertices (Fig. \ref{Efimov2}). 
\begin{figure}[!ht]
 \begin{center}
\includegraphics[scale=0.9]{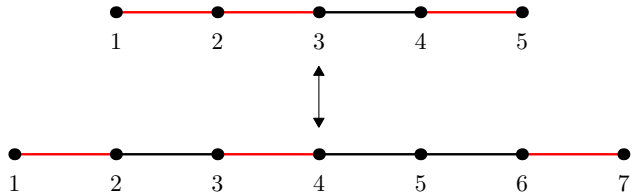}
 \caption{Transition between the path graphs $P_5$ and $P_7$ with $3$ selected (in red) edges}\label{Efimov2}
 \end{center}
\end{figure}
It is not hard to see that such a procedure and its inverse establishes a one-to-one correspondence
between samples of $k$ edges in a path graph with $n$ vertices and samples of $k$ edges, any two of which have no common vertices, 
in a path graph with $n+k-1$ vertices.
Therefore,  $C_{n-1}^k=E_{n+k-1}^k$.  It follows that $E_n^k=C_{n-k}^k$.
\end{proof} 

Now we can formulate and prove the main result.

\begin{thm}
Let $R$ be a commutative associative ring with unit element, and $a,b\in R$.
Consider a symmetric matrix $T_{a,b}$ of order $2m$   
whose elements on the main diagonal are zero, the elements on the subdiagonal and on the superdiagonal are equal to $a$, 
and all others elements are equal to $b$:
\begin{equation}\label{26.10.18_1}
T_{a,b}=
\left(
\begin{array}{cccc}
0&a&&b\\
a&\ddots&\ddots&\\	
&\ddots&\ddots&a\\
b&&a&0
\end{array}
\right).
\end{equation}
Assuming $0^0=1$, the following equality holds:
\begin{equation}\label{17.10.18_2}
  \mathrm{Hf}(T_{a,b})=\sum_{k=0}^m (a-b)^{m-k}b^k\frac{(m+k)!}{k!(m-k)!2^k}\ .
\end{equation}
\end{thm}
\begin{proof}
Let $J_q$ denote a symmetric matrix of order $2m$
whose elements on the main diagonal are zero, and all others elements equal $q$.
Directly from the definition of the hafnian it follows that 
\begin{equation}\label{09.10.18_1}
 \mathrm{Hf}(J_q)=q^m\frac{(2m)!}{m!2^m}.
\end{equation}
Let $U_q$ denote the matrix of order $2m$ 
whose elements on the subdiagonal and on the superdiagonal are equal to $q$, and all others elements are zeros. 
Since $T_{a,b}=J_b+U_{a-b}$, we can write the following chain of equalities 
using the formulas (\ref{16.10.18_1}), (\ref{24.11_3}), and (\ref{09.10.18_1}):
\begin{equation}\label{16.10.18_2}
\begin{split}
  \mathrm{Hf}(T_{a,b})=\mathrm{Hf}(J_b+U_{a-b})=&\sum_{k=0}^m \sum_{\alpha\in Q_{2k,2m}}\mathrm{Hf}(J_b[\alpha])\mathrm{Hf}(U_{a-b}(\alpha))=\\
	&=\sum_{k=0}^m (a-b)^{m-k}b^k\frac{(2k)!}{k!2^k}\sum_{\alpha\in Q_{2k,2m}}\mathrm{Hf}(U_1(\alpha)).
\end{split}
\end{equation}
Here we use the fact that $J_b[\alpha]$ has the same form as the initial matrix $J_b$,
i.e.,  $J_b[\alpha]$ is a symmetric matrix of order $2k$ whose elements on the main diagonal are zeros, and all others elements are equal to $b$.

Let $u_{ij}$ denote the elements of $(0,1)$-matrix $U_1$. 
If $\alpha\in Q_{2k,2m}$,  then the matrix $U_1(\alpha)$ has the order $2m-2k$, and, by the definition,
\begin{equation}\label{17.10.18_1}
  \sum_{\alpha\in Q_{2k,2m}}\textrm{Hf}(U_1(\alpha))=\sum_{\alpha\in Q_{2k,2m}}\sum_P u_{i_1i_2}\dots u_{i_{2m-2k-1}i_{2m-2k}},
\end{equation}
where the inner sum runs over all partitions $P$ of the set $\{1,2,\dots,2m\}\backslash\alpha$ 
into disjoint pairs $(i_1i_2),\dots,(i_{2m-2k-1}i_{2m-2k})$  
up to an order of pairs and an order of elements in each pair.
Since only the elements $u_{i,i+1}$ and $u_{i+1,i}$ of $U_1$ are equal to $1$,
the term $u_{i_1i_2}\dots u_{i_{2m-2k-1}i_{2m-2k}}$ in the sum (\ref{17.10.18_1})
equals $1$ if and only if each pair in the partition $(i_1i_2),\dots,(i_{2m-2k-1}i_{2m-2k})$ 
consists of two neighboring indexes, otherwise this term equals zero.
It follows that the sum (\ref{17.10.18_1}) equals the number of different ways
to select in the set $\{1,2,\dots,2m\}$ a collection of $m-k$ disjoint pairs $(i_1,i_1+1),\dots,(i_{m-k},i_{m-k}+1)$, 
up to an order of pairs. 
And this is nothing else than the number of ways to select $m-k$ edges in the path graph $P_{2m}$
so that any two edges do not have common vertices. 
By Proposition \ref{08.02_1} this number equals $C_{2m-m+k}^{m-k}=C_{m+k}^{m-k}$.
Thus substituting this value in (\ref{16.10.18_2}), we get the desired expression:
$$
 \mathrm{Hf}(T_{a,b})=\sum_{k=0}^m (a-b)^{m-k}b^k\frac{(2k)!}{k!2^k}C_{m+k}^{m-k}=\sum_{k=0}^m (a-b)^{m-k}b^k\frac{(m+k)!}{k!(m-k)!2^k}.
$$
\end{proof}

Now we make a few comments about the result.
Let $a,b$ be non-negative integers, 
and let $G_{a,b}$ denote the graph with the adjacency matrix $T_{a,b}$.
If one draws $G_{a,b}$ in the form of an arc diagram, then the neighboring vertices will be connected 
by $a$ arcs while all the other pairs of vertices by $b$ arcs. 
In this case the formula (\ref{17.10.18_2}) expresses the number of perfect matchings of the graph $G_{a,b}$.

Let $a=0$, $b=1$. The graph $G_{0,1}$ with $n$ vertices is the complement of the path graph $P_n$ (Fig. \ref{Efimov3}).
And it is not hard to see that perfect matchings of the graph $G_{0,1}$ with $2m$ vertices
are loopless linear chord diagrams with $m$ chords considered in \cite{Krasko}.
From (\ref{17.10.18_2}) we get 
\begin{equation}\label{16.04.19_1}
 \mathrm{Hf}(T_{0,1})=\sum_{k=0}^m (-1)^{m-k}\frac{(m+k)!}{k!(m-k)!2^k}.
\end{equation}
Thus, the formula (\ref{16.04.19_1}) expresses the number of loopless linear chord diagrams with $m$ chords (the sequence $A278990$ in \cite{oeis}). 
\begin{figure}[!ht]
 \begin{center}
 \includegraphics[scale=0.6]{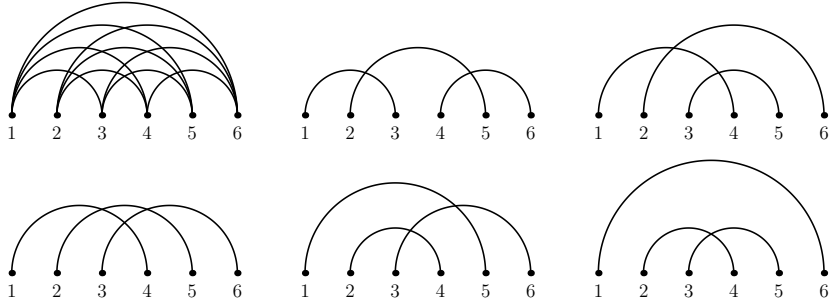}
 \caption{The graph $G_{0,1}$ with six vertices and all its perfect matchings}\label{Efimov3}
 \end{center}
\end{figure}

Consider now the graph $G_{2,1}$. 
If one calculates by (\ref{17.10.18_2}) the hafnian of its adjacency matrix for consecutive $m$, starting with $m=1$, we get the sequence:
$$
2,7,37,266,2431,27007,\dots 
$$ 
Its $k$-th member equals the number of perfect matchings of the graph 
$G_{2,1}$ with $2k$ vertices (Fig. \ref{Efimov4}). 
\begin{figure}[!ht]
 \begin{center}
 \includegraphics[scale=0.6]{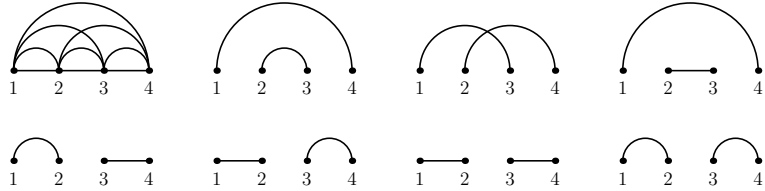}
 \caption{The graph $G_{2,1}$ with four vertices and all its perfect matchings}\label{Efimov4}
 \end{center}
\end{figure}
Note that this sequence has the notation $A001515$ in \cite{oeis},
however  its description does not contain the  interpretation given here.

Recall (see \cite{Bessel1}, \cite{Bessel2}) that the Bessel polynomial of degree $m$ is a polynomial of the form: 
$$
y_m(x)=\sum_{k=0}^m \frac{(m+k)!}{k!(m-k)!}\left(\frac{x}{2}\right)^k.
$$
It follows from (\ref{17.10.18_2}) that the hafnian of the matrix $T_{b+1,b}$ of order $2m$ equals the value of the Bessel polynomial of degree $m$ at $x=b$:
$$
\mathrm{Hf}(T_{b+1,b})=y_m(b).
$$
This is a rather curious and unexpected fact, the explanation of which is not clear so far.

%------------------------------------------------------------------------------------------------------

\section*{Conclusion}

We obtained a simple analytical formula 
for efficient exact calculation of the hafnian of Toeplitz matrices of the special type (\ref{26.10.18_1}).
Based on this formula, it is not hard to write an algorithm which calculates the hafnian of a matrix of order $n$  in $O(n)$ time. 
The resulting formula can be used to find the number of perfect matchings of special graphs, 
for example, complements of the path graphs.
Along the way, we established an intresting connection between the hafnian of Toeplitz matrices and the Bessel polynomials. 
This connection requires more detailed study. 
Also, one could try, using the above methods, to find effective analitical formulas for calculating hafnians
of other types of Toeplitz matrices.

%-------------------------------------------------------------------------------------------------------

\end{document}